\documentclass[12pt]{amsart}

\usepackage{amssymb, amsthm, hyperref}
\usepackage{amsmath}
\usepackage{enumitem, amsfonts, amsxtra,latexsym,epsfig, color}
\usepackage{cases}
\usepackage{thm-restate}
\usepackage{mathrsfs}

\usepackage{hyperref}
\usepackage{tikz-cd}

\newtheorem {theorem}{Theorem}[section]
\newtheorem {lemma} [theorem] {Lemma}
\newtheorem {proposition} [theorem] {Proposition}

\newtheorem* {claim*}{Claim}

\newtheorem* {subclaim*}{Subclaim}

\newtheorem {question} [theorem] {Question}

\theoremstyle{definition}
\newtheorem{remark}[theorem]{Remark}
\newtheorem {definition} [theorem] {Definition}

\newcommand{\rightQ}[2]{\left.\raisebox{.2em}{$#1$}\middle/\raisebox{-.2em}{$#2$}\right.}
\newcommand{\leftQ}[2]{\left.\raisebox{-.2em}{$#2$}\middle\backslash\raisebox{.2em}{$#1$}\right.}

\newcommand{\mc}{\mathcal}
\newcommand{\from}{\colon\thinspace}

\newcommand{\bZ}{\mathbb{Z}}

\newcommand{\bR}{\mathbb{R}}

\newcommand{\bS}{\mathbb{S}}

\newcommand{\llangle}{\langle\negthinspace\langle}
\newcommand{\rrangle}{\rangle\negthinspace\rangle}

\begin{document}

\title{Special IMM groups}

\author[D. Groves]{Daniel Groves}
\address{Department of Mathematics, Statistics, and Computer Science,
University of Illinois at Chicago,
322 Science and Engineering Offices (M/C 249),
851 S. Morgan St.,
Chicago, IL 60607-7045}
\email{groves@math.uic.edu}

\author[J.F. Manning]{Jason Fox Manning}
\address{Department of Mathematics, 310 Malott Hall, Cornell University, Ithaca, NY 14853}
\email{jfmanning@cornell.edu}

\thanks{The first author was supported in part by NSF Grant DMS-1904913.  The second author thanks the Simons Foundation for support under grant (\#524176, JFM)}

\begin{abstract}
Italiano--Martelli--Migliorini recently constructed hyperbolic groups which have non-hyperbolic subgroups of finite type.    Using a closely related construction, Llosa Isenrich--Martelli--Py constructed hyperbolic groups with subgroups of type $\mc{F}_3$ but not $\mc{F}_4$.  We observe that these hyperbolic groups can be chosen to be special in the sense of Haglund--Wise.
\end{abstract}

\date{\today}

\maketitle
\section{Introduction}
In \cite{IMM5} Italiano--Martelli--Migliorini construct the first examples of hyperbolic groups containing non-hy\-per\-bol\-ic subgroups of finite type.  This answered a well known question which had been open for many years (see \cite[Question 1.1]{bestvina:problems2004}, \cite[Question 7.2]{brady:subgroups}, \cite[$\S7$]{jnw}).  In this note, we point out that these examples  can be constructed so as to be \emph{special}, in the sense of Haglund--Wise \cite{HW08}.  In particular, these examples embed in right-an\-gled Artin groups, and thus inherit linearity, residual finiteness, and geometric subgroup separability properties from those groups.  Also, being nonabelian, these groups surject a nonabelian free group \cite[Theorem 1.1]{bregman}.  We also explain how Llosa Isenrich--Martelli--Py's examples in \cite[Theorem 1]{LIMP} of hyperbolic groups with subgroups of type $\mathscr{F}_3$ but not $\mathscr{F}_4$ can be chosen to be special.

The hyperbolic manifold examples constructed in \cite{IMM5} and \cite{LIMP} are commensurable with right-an\-gled Coxeter groups, so their virtual spe\-cial-ness is immediate.  However the manifolds constructed are only finite volume and not compact, so \cite{IMM5} and \cite{LIMP} must apply Dehn filling techniques to obtain hyperbolic groups with the appropriate subgroup structure.  By using a version of the Malnormal Special Quotient Theorem due to Wise \cite{WisePUP}, we will explain how this Dehn filling can be chosen to preserve virtual spe\-cial-ness.

\begin{theorem}\label{thm:main}
  There are infinitely many (pairwise non-isomorphic) special hyperbolic groups, each of which fits into a short exact sequence
  \[ 1 \to K \to G \to \bZ \to 1 \]
  so that $K$ is non-hyperbolic and finite type.
\end{theorem}
These are particular cases of the construction in \cite{IMM5}.  
Thus the hyperbolic group $G$ is the fundamental group of a negatively curved (hence aspherical) pseudo-manifold of dimension $5$, and $K$ is the fundamental group of an aspherical pseudo-manifold of dimension $4$.  

\begin{theorem}\label{thm:main2}
 There are infinitely many (pairwise non-isomorphic) special hyperbolic groups, each of which fits into a short exact sequence
 \[	1 \to K \to G \to \bZ \to 1 \]
 so that $K$ is of type $\mathscr{F}_3$ but not of type $\mathscr{F}_4$.
\end{theorem}
These are particular cases of the construction in \cite{LIMP}, and thus $G$ is here the fundamental group of a negatively curved pseudo-manifold of dimension $8$.  Note that in order to obtain infinitely many examples in \cite[Theorem 1]{LIMP}, a delicate argument about perturbing the fibration was required.  However, once we have one example which is residually finite, we will obtain infinitely many examples by passing to finite index subgroups.  Simplicial volume will be used to prove there are infinitely many different examples.  Of course, these groups are all commensurable.

\subsection*{Acknowledgements}
Thanks to Gilbert Levitt for pointing us to Brinkmann's thesis \cite{brinkmann:thesis}.  Thanks to Colby Kelln for noticing a confusing typo in an earlier version of this note.  Thanks to the referee for pointing out a minor error in an earlier version.

\section{The examples}
In this section we recall some important facts about the manifold examples constructed by Italiano--Martelli--Migliorini in \cite{IMM8} and \cite{IMM5} and by Llosa Isenrich--Martelli--Py in \cite{LIMP}.  These examples satisfy the following hypotheses (after possibly passing to a finite sheeted cover).
\begin{enumerate}[label = (H\arabic*)]
\item\label{itm:finitevolume} $M$ is an orientable non-compact hyperbolic $n$--manifold of finite volume.
  \item\label{itm:toralcusp} $M$ contains disjoint open cusp neighborhoods $C_1,\ldots C_k$ so that $\overline M = M - \bigcup C_i$ is a compact manifold whose boundary consists of flat $(n-1)$--dimensional tori. 
\item\label{itm:relhyp}  If $G = \pi_1 M$ and $\mc{P} = \{\pi_1C_i\}_{i = 1}^k$, then $(G,\mc{P})$ is relatively hyperbolic.
\item\label{itm:algfiber} There is a surjective map $f \from G \to \bZ$ which is nontrivial on each $P\in\mc{P}$, and which has finitely generated kernel.
\item\label{itm:racg} $G$ is a finite index subgroup of a right-an\-gled Coxeter group $\Gamma$ and $\overline{M}$ is homotopy equivalent to the cube complex $\leftQ{D}{G}$ where $D$ is the Davis complex for $\Gamma$.
\end{enumerate}
We make a few comments on these features.

  The hypothesis in~\ref{itm:toralcusp} that the boundary consists of tori may be stronger than necessary for what we do.  In Section 2 of \cite{IMM5}, a fibered hyperbolic $5$--manifold is considered which does \emph{not} have torus cusp cross sections.  This $5$--manifold was previously considered by Ratcliffe--Tschantz \cite{RatcliffeTschantz04} and is the smallest known hyperbolic $5$--manifold.  We do not consider this example in this paper, though it should be possible to perform similar constructions on such an example.

  Hypothesis~\ref{itm:relhyp} follows from~\ref{itm:finitevolume} (see~\cite{farb:relhyp}).  
  
  Hypothesis~\ref{itm:algfiber} says that the map $f$ is an \emph{algebraic fibering} of $\pi_1M$.  We are most interested in the cases when $\ker(f)$ enjoys additional finiteness properties.  The examples in~\cite{LIMP} have the property that $\ker(f)$ is type $\mc{F}_3$ but not type $\mc{F}_4$.  The examples in~\cite{IMM5} have the property that $\ker(f)$ has type $\mc{F}$.

  Hypothesis~\ref{itm:racg} is immediate from the construction.  It follows that $\overline{M}$ is homotopy equivalent to a virtually special cube complex \cite{HW08}.

  We fix some $M$ as in Hypotheses~\ref{itm:finitevolume}-\ref{itm:racg} for the rest of the section.

\begin{definition}
 Given the relatively hyperbolic pair $(G,\mc{P})$ and a collection $\mc{N} = \{ N_P \lhd P \mid P \in \mc{P} \}$ the \emph{Dehn filling of $G$ along $\mc{N}$} is the group
 \[	G\left(\mc{N} \right) = \rightQ{G}{\left\llangle \cup_{N \in \mc{N}} N\right\rrangle}	.	\]
\end{definition}

The relatively hyperbolic version of the Malnormal Special Quotient Theorem \cite[Theorem 15.6]{WisePUP}, applied using Hypotheses~\ref{itm:relhyp} and~\ref{itm:racg}, yields the following.
\begin{proposition} \label{prop:MSQT_RH}
  There are finite index normal subgroups $\{\dot{P}\lhd P\mid P\in \mc{P}\}$ so that, for any collection $\mc{N} = \{N_P \lhd P\mid P\in \mc{P}\}$ with (i) $N_P<\dot{P}$ for each $P$; and (ii) each $P/N_P$ virtually cyclic, the Dehn filling $G(\mc{N})$ is hyperbolic and virtually compact special.  Moreover $\ker(G\to G(\mc{N}))\cap P = N_P$ for each $P\in \mc{P}$.
\end{proposition}
Using Hypotheses~\ref{itm:toralcusp} and~\ref{itm:algfiber} the following lemma says we can find appropriate $\mc{N}$ as above which moreover interact nicely with the algebraic fibering $f$.
\begin{lemma}\label{lem:fillingkernels}
  There is a collection $\mc{N} = \{ N_P \lhd P \mid P \in \mc{P} \}$ which satisfy the following conditions for each $P \in \mc{P}$:
\begin{enumerate}
\item\label{itm:msqt} $N_P \le \dot{P}$;
\item\label{itm:virtZ} $P/N_P$ is virtually cyclic;
\item\label{itm:inkernel} $N_P <\ker(f)$;
\item\label{itm:2pi} $N_P - \{1\}$ contains no element whose representative as a curve on the boundary of $\overline{M}$ has length $\le 2\pi$.
\end{enumerate}
\end{lemma}
\begin{proof}
  Since the $P$ are all isomorphic to $\bZ^{n-1}$, and there are only finitely many closed geodesics on each boundary component of $\overline{M}$ of length at most $2\pi$, finding $N_P$ which satisfy the above conditions is straightforward.
\end{proof}
The next result uses all the hypotheses~\ref{itm:finitevolume}-\ref{itm:racg}.
\begin{proposition}\label{prop:summary}
  Let $\mc{N}$ be a collection satisfying the conclusions of Lemma~\ref{lem:fillingkernels}, and let $\overline{G} = G(\mc{N})$.
  Then the algebraic fibering $f\from G\to \bZ$ is equal to $\overline{f}\circ q$, where $q$ is the Dehn filling map $q\from G\to\overline{G}$ and $\overline{f}\from \overline{G}\to \bZ$ is an algebraic fibering.
  Moreover $\overline{G}$ contains a finite index normal subgroup $\overline{G}_0$ so that the following hold.
  \begin{itemize}
  \item $\overline{G}_0$ is special.
  \item $\overline{G}_0$ is the fundamental group of a closed orientable negatively curved  pseudomanifold of dimension $n$.
  \end{itemize}
\end{proposition}
\begin{proof}
  By Proposition~\ref{prop:MSQT_RH} and the choice of $N_P \le \dot{P}$, we see that $\overline{G}$ is a virtually special hyperbolic group.  Let $\overline{K}$ be the image of $\ker(f)$ in $\overline{G}$ under the natural quotient map $q \from G \to \overline{G}$.  Since the kernel of $q$ is normally generated by elements in $\ker(f)$, the algebraic fibering factors through $q$ as desired, and there is a short exact sequence
  \[	1 \to \overline{K} \to \overline{G} \stackrel{\overline{f}}{\to} \bZ \to 1	.	\]
  As the image of a finitely generated group, $\overline{K}$ is finitely generated.

Let $\overline{G}_0 \le \overline{G}$ be a finite index normal special subgroup, which is in particular torsion-free, and let $G_0 = q^{-1}\left(\overline{G}_0\right)$.  Since $G_0$ has finite index in $G$, there is a finite-sheeted regular covering $\widetilde{M} \to \overline{M}$ so that $G_0 = \pi_1\left(\widetilde{M}\right)$.

Let $P_0$ be a cusp subgroup of $G_0$, which is of the form $P_0 = P^g \cap G_0$ for some $P \in \mc{P}$ and $g \in G$.  By the last sentence of Proposition~\ref{prop:MSQT_RH}, the intersection of $\ker(q)$ with $P$ is $N_P$.
Let $N_0 = N_P^g \cap G_0$ be the intersection of $P_0$ with $\ker(q)$.  Then $P_0 / N_0$ is virtually cyclic and embeds in the torsion-free $\overline{G}_0$, so $P_0 / N_0 \cong \bZ$ and $N_0$ is a direct summand of $P_0$.
Moreover, $\ker(f) \cap P_0 \le N_0$, so the map $P_0 \to P_0 / N_0 \cong \bZ$ factors through the map  $P_0 \to \bZ$ induced by the fibration.  Since $\bZ$ is torsion free we have $\ker(f) \cap P_0 = N_0$.

We now have a short exact sequence
\[	1 \to \overline{K}_0 \to \overline{G}_0 \to \bZ \to 1	, \]
where $\overline{K}_0 = \overline{K} \cap \overline{G}_0$.

Each cusp group $P_0<G_0$ is the fundamental group of some boundary component $T_0\cong T^{n-1}$, fibered by totally geodesic tori of dimension $(n-2)$ whose fundamental groups are all equal to $N_0$.  On the level of topology, the map $q|_{G_0}$ is realized by coning all these $(n-2)$--dimensional tori to points, in other words attaching a copy of $C(T^{n-2})\times S^1$ to each boundary component.  The resulting space is a closed $n$--dimensional pseudomanifold, which is a manifold away from a finite collection of circles.  In~\cite{fm:cat0}, it is shown that under the condition~\eqref{itm:2pi} of Lemma~\ref{lem:fillingkernels}, this closed pseudo-manifold can be given a locally CAT$(-1)$ metric.
\end{proof}

The following summarizes two theorems from \cite{LIMP}; it uses the hypotheses \ref{itm:finitevolume}-\ref{itm:algfiber}.
\begin{theorem}\cite[Theorems 3 and 5]{LIMP}\label{thm:limp}
  Let $\overline{f}\from G(\mc{N})\to\bZ$ be as in Proposition~\ref{prop:summary}.  If $f_0 = \overline{f}|_{\overline G_0}$ and $\ker(f)$ is $\mc{F}_k$ for some $k\ge 0$, then $\ker(f_0)$ is also $\mc{F}_k$.  If $n$ is even and $\ker(f)$ is not $\mc{F}_{n/2}$, then neither is $\ker(f_0)$.
\end{theorem}

\section{Proofs of Theorems~\ref{thm:main} and \ref{thm:main2}}

\begin{proof}[Proof of Theorem~\ref{thm:main}]
  The construction in~\cite{IMM5} begins by constructing an $M$ satisfying the hypotheses \ref{itm:finitevolume}-\ref{itm:racg}, which additionally has the property that $f = \rho_\ast$ where $\rho \from M \to \bS^1$ is a fiber bundle map.  This fibration restricts on each boundary component to a fibration by totally geodesic sub-tori (see \cite[\S 1.13]{IMM5}).

  Choosing a collection $\mc{N}$ of filling kernels as in Lemma~\ref{lem:fillingkernels}, we obtain a map $q\from G\to \overline{G}$ and a finite index subgroup $\overline{G}_0\lhd\overline{G}$ as in the conclusion of Proposition~\ref{prop:summary}.  As in the proof of that proposition we set $G_0 = q^{-1}(\overline{G}_0)$ and let $\widetilde{M}$ be the corresponding finite-sheeted cover.

  On the level of topology, the map  $q|_{G_0} \from G_0 \to \overline{G}_0$ can be obtained from coning the geodesic $3$--tori fibers in the boundary components of $\widetilde{M}$ to points, to obtain a space $\widehat{M}$ with $\pi_1 \left( \widehat{M} \right) \cong \overline{G}_0$.  Moreover, there is an induced fibration $\widehat{M} \to \bS^1$ with fiber $\widehat{F}$ the fiber of $\widetilde{M} \to \bS^1$ with boundary $3$--tori coned to points.

  This is exactly the construction described in \cite[$\S3$]{IMM5} applied to the manifold $\widetilde{M}$.  As in \cite[$\S3$]{IMM5}, there is an infinite cyclic cover of the aspherical pseudomanifold $\widehat M$ which is homeomorphic to $\widehat{F}\times \bR$, and so $\widehat{F}$ is aspherical.  In particular $\pi_1\widehat{F}$ has finite type.  Moreover the arguments of \cite[$\S3$]{IMM5} show $\pi_1\widehat{F}$ is not hyperbolic.  
 \footnote{Here is an alternative argument that $\pi_1\widehat{F}$ is not hyperbolic.  Since $\pi_1\hat M$ is torsion-free hyperbolic, $\pi_1\widehat{F}$ admits an infinite order atoroidal outer automorphism, namely the monodromy of the bundle.  If $\pi_1\widehat{F}$ were hyperbolic, a result of Sela \cite[Corollary 1.10]{sela:srII} would imply that $\pi_1\widehat{F}$ was a free product of free and surface groups -- see \cite[Chapter 2]{brinkmann:thesis}.  On the other hand $H_4(\pi_1\widehat{F}) = H_4(\widehat{F}) = \bZ$, since $\widehat{F}$ is a closed orientable $4$--dimensional pseudomanifold.}

  To get infinitely many groups as in Theorem~\ref{thm:main}, we may pass to subgroups of $\overline{G}_0$ of larger and larger finite index, which exist because $\overline{G}_0$ is special and hence residually finite.  To distinguish these, we use simplicial volume.  (This is a homotopy invariant, so the simplicial volume of $\overline{G}_0$ is the same as the simplicial volume of $\widehat{M}$.)
  Since $\widehat{M}$ is a negatively curved closed pseudomanifold, its simplicial volume is positive (see \cite{Ya97} for an explicit estimate).  Simplicial volume is multiplicative under covers (see \cite[Proposition 7.2]{Frigerio}), so the simplical volumes of our finite index subgroups of $\overline{G}_0$ go to infinity.  This establishes Theorem~\ref{thm:main}.
\end{proof}

\begin{proof}[Proof of Theorem~\ref{thm:main2}]
  In~\cite{LIMP} Llosa Isenrich--Martelli--Py give an example of a hyperbolic $8$--manifold satisfying the hypotheses \ref{itm:finitevolume}-\ref{itm:racg} and with the additional property that $\ker(f)$ has property $\mc{F}_3$ but not $\mc{F}_4$.  Theorem~\ref{thm:limp} implies that these finiteness properties persist in Dehn fillings as described in Proposition~\ref{prop:summary}.  We therefore obtain a special hyperbolic closed pseudo-manifold group with a subgroup of type $\mc{F}_3$ but not of type $\mc{F}_4$.  As in the proof of~\ref{thm:main}, since this group is residually finite, we can find infinitely many examples by passing to finite index subgroups.
\end{proof}

\section{Questions and remark}

\begin{question}
  We see that these groups are QCERF.  Are they LERF?  
\end{question}

\begin{question}
 In the language of Lubotzky--Manning--Wilton \cite{LMW} we prove that a \emph{positive fraction} of $2\pi$--fillings as above are virtually special.  Is this true for all sufficiently long $2\pi$--fillings?
\end{question}

\begin{question}
 Is it possible to obtain infinitely many commensurability classes in the conclusions of Theorems~\ref{thm:main} and/or \ref{thm:main2}?
\end{question}

\begin{remark}
In \cite{Fujiwara}, Fujiwara starts with the fibered $5$--manifold of \cite{IMM5} and does a manifold filling to obtain an aspherical manifold (with relatively hyperbolic fundamental group) which fibers over the circle.  In a subsequent paper, we will explain that Fujiwara's examples can be constructed to have residually finite fundamental group.
\end{remark}

\end{document}